\documentclass[11pt,a4paper]{article}
\usepackage[utf8]{inputenc}
\usepackage[english]{babel}
\usepackage[T1]{fontenc}
\usepackage{amsmath}
\usepackage{amsfonts}
\usepackage{amssymb}
\usepackage{amsthm}
\usepackage{mathrsfs}
\usepackage{dsfont}
\usepackage[left=2cm,right=2cm,top=3cm,bottom=3cm]{geometry}
\usepackage{accents}
\usepackage{comment}
\usepackage{enumitem}

\RequirePackage[colorlinks,citecolor=blue,urlcolor=blue]{hyperref}

\usepackage{tikz}
\usepackage{pgfplots}
\usetikzlibrary{arrows,arrows.meta,calc,quotes,angles,shapes,plotmarks}

\newtheorem{theorem}{Theorem}[section]
\newtheorem{lemma}[theorem]{Lemma}
\newtheorem{proposition}[theorem]{Proposition}

\newtheorem{remark}[theorem]{Remark}

\numberwithin{equation}{section}

\DeclareMathOperator{\argmin}{argmin}

\allowdisplaybreaks

\providecommand{\keywords}[1]
{
	\textbf{\textit{Keywords--}} #1
}

\makeatletter
\newcommand{\owntag}[2][\relax]{
  \ifx#1\relax\relax\def\owntag@name{#2}\else\def\owntag@name{#1}\fi
  \refstepcounter{equation}\tag{\theequation, #2}%
  \expandafter\ltx@label\expandafter{eq:\owntag@name}%
  \def\@currentlabel{\theequation, #2}\expandafter\ltx@label\expandafter{Eq:\owntag@name}%
  \def\@currentlabel{#2}\expandafter\ltx@label\expandafter{tag:\owntag@name}%
}
\makeatother

\title{A note on $L^1$-Convergence of the Empiric Minimizer for unbounded functions with fast growth}
\author{Pierre Bras\footnote{Sorbonne Universit\'e, Laboratoire de Probabilit\'es, Statistique et Mod\'elisation, UMR 8001, case 188, 4 pl. Jussieu, F-75252 Paris Cedex 5, France. E-mail: \texttt{pierre.bras@sorbonne-universite.fr}}}
\date{}

\begin{document}

\maketitle

\begin{abstract}
For $V : \mathbb{R}^d \to \mathbb{R}$ coercive, we study the convergence rate for the $L^1$-distance of the empiric minimizer, which is the true minimum of the function $V$ sampled with noise with a finite number $n$ of samples, to the minimum of $V$. We show that in general, for unbounded functions with fast growth, the convergence rate is bounded above by $a_n n^{-1/q}$, where $q$ is the dimension of the latent random variable and where $a_n = o(n^\varepsilon)$ for every $\varepsilon > 0$.
We then present applications to optimization problems arising in Machine Learning and in Monte Carlo simulation.
\end{abstract}

\keywords{Stochastic Optimization, Empirical measure, Empiric minimizer}

\section{Introduction}

Stochastic algorithms are powerful tools to solve complex optimization problems in high dimension. We consider a function $V : \mathbb{R}^d \to (0,\infty)$ to be minimized and we assume that there exists a random variable $Z$ with values in $\mathbb{R}^q$ and a function $v : \mathbb{R}^d \times \mathbb{R}^q \to \mathbb{R}^+$ such that
$$ \forall x \in \mathbb{R}^d, \ \mathbb{E}[v(x,Z)] = V(x) .$$
Furthermore, we focus on the case where $V$ cannot or is too costly to be computed exactly and we need to rely on noisy observations from $v$ to estimate $V$. To minimize $V$, we proceed as follows: we first draw i.i.d. samples $Z_1$, $\ldots$, $Z_n$ and then instead of $V$ we minimize the empirical version
$$ V_n(x) = \frac{1}{n} \sum_{i=1}^n v(x,Z_i) ,$$
typically with stochastic gradient descent algorithms.
Such procedures appear in Machine Learning related problems, where the original data set is too large and where we only use a (still large) subset of it, and in Monte Carlo simulation where we approximate an expectation by an empirical average.

A question that naturally arises is to prove the convergence of the empirical minimum to the true minimum. Noting $X_n^\star$ a minimizer of $V_n$ and $x^\star$ a minimizer of $V$, we give convergence rates for $\mathbb{E}|V_n(X_n^\star)-V(x^\star)|$, where the expectation is taken on the sampling of $Z_1$, $\ldots$, $Z_n$.  It turns out to be of order $a_n n^{-1/q}$, where $a_n=o(n^\varepsilon)$ for every $\varepsilon>0$. In the case where the minimum of $V$ is reached at only one point, we also give convergence rates for $\mathbb{E}|X_n^\star-x^\star|^2$. To this end, we rely on \cite{fournier2014} giving the convergence rate of the empirical measure $(1/n) \sum_{i=1}^n \delta_{Z_i}$ to the law of $Z$ for the $L^1$-Wasserstein distance and showing that in general, the convergence rate is of order $n^{-1/q}$ and that this bound is sharp, so we cannot expect a better convergence rate for our problem in general.

The convergence of the empiric risk minimizer was notably studied in \cite{vapnik1971} for binary classification problems, where it is shown that the quantity $|V_n(X_n^\star)-V(x^\star)|$ is closely related to $\sup_{x \in \mathbb{R}^d}|V_n(x)-V(x)|$ (see Lemma \ref{lemma:1}) and if $\mathcal{V} := \lbrace v(x, \cdot) : \ x \in \mathbb{R}^d \rbrace$ is a class of $\lbrace 0,1 \rbrace$-valued functions and that the Vapnik-Chervonenkis dimension of $\mathcal{V}$, which is an integer taking into account the combinatorial properties of $\mathcal{V}$, is finite, then $ \mathbb{E}\sup_{x \in \mathbb{R}^d}|V_n(x)-V(x)|$ converges to zero with rate $1/\sqrt{n}$ (see \cite[Lemma2.4]{bartlett2006}).
\cite{bartlett2006} and \cite{bartlett2010} extend these results to bounded, Bernstein, star-shaped class of functions using concentration inequalities.
However, this question has not been studied in the setting of Monte Carlo simulation.

\cite{biau2008} and \cite{liu2020} study the risk of the optimal quantization of a random vector by $k$-means clustering and prove an upper bound in $O(\log(n)/n)$, however these results cannot be directly applied for the convergence rate of the empirical minimizer.

In the present paper we prove the convergence rate of $n^{-1/q}$ for more general, unbounded classes of functions with fast growth and where we do not assume convexity.

The article is organized as follows. In Section \ref{sec:main_results}, we state our assumptions and then give our main results. In Section \ref{sec:applications}, we give classic examples in Machine Learning and Monte Carlo simulations where our assumptions are verified and where our results can be applied.

\bigskip

\textsc{Notations}

We endow the space $\mathbb{R}^d$ with the canonical Euclidean norm denoted by $| \boldsymbol{\cdot} |$. For $x \in \mathbb{R}^d$ and for $R>0$, we denote $\textbf{B}(x,R) = \lbrace y \in \mathbb{R}^d : \ |y-x| \le R \rbrace$.

For $f :\mathbb{R}^d \rightarrow \mathbb{R}$ such that $\min_{\mathbb{R}^d}(f)$ exists, we denote $\text{argmin}(f) = \left\lbrace x \in \mathbb{R}^d : \ f(x) = \min_{\mathbb{R}^d}(f) \right\rbrace$.

If $f : \mathbb{R}^d \to \mathbb{R}$ is Lipschitz-continuous, we denote by $[f]_{\text{Lip}}$ its Lipschitz constant. We often consider the Lipschitz constant of $f$ restricted to some set $\mathcal{E} \subset \mathbb{R}^d$ that we denote $[f_{|\mathcal{E}}]_{\text{Lip}}$.

We say that $f$ is coercive if $\lim_{|x| \rightarrow \infty} f(x) = + \infty$.

We denote the $L^p$-Wasserstein distance between two distributions $\pi_1$ and $\pi_2$ on $\mathbb{R}^d$:
$$ \mathcal{W}_p(\pi_1, \pi_2) = \inf \left\lbrace \left(\int_{\mathbb{R}^d} |x-y|^p \pi(dx,dy) \right)^{1/p} : \ \pi \in \mathcal{P}(\pi_1,\pi_2) \right\rbrace ,$$
where $\mathcal{P}(\pi_1,\pi_2)$ stands for the set of probability distributions on $(\mathbb{R}^d \times \mathbb{R}^d, \mathcal{B}or(\mathbb{R}^d)^{\otimes 2})$ with respective marginal laws $\pi_1$ and $\pi_2$.

For $x \in \mathbb{R}^d$, we denote by $\delta_x$ the Dirac mass at $x$.

For a set $\mathcal{E} \subset \mathbb{R}^d$, we define $\text{Radius}(\mathcal{E}) := \sup \lbrace |x| : x \in \mathcal{E} \rbrace$.

In this paper, we use the notation $C$ and $C'$ to denote positive constants, which may change from line to line.

\section{Main results}
\label{sec:main_results}

Let $V: \mathbb{R}^d \to (0,\infty)$ be $\mathcal{C}^0$ and coercive i.e. $V(x) \to +\infty$ as $|x| \to \infty$. This guarantees that $\min_{\mathbb{R}^d} V$ exists, let us denote $V^\star := \min_{\mathbb{R}^d} V > 0$. We assume that there exists a random variable $Z$ defined on some probability space $(\Omega, \mathcal{A}, \mathbb{P})$ with law $\mu$ taking its values in $\mathbb{R}^q$ and some function $v : \mathbb{R}^d \times \mathbb{R}^q \to \mathbb{R}^+$ being $\mathcal{C}^0$ such that
\begin{equation}
\forall x \in \mathbb{R}^d, \ \mathbb{E}[v(x,Z)] = V(x) .
\end{equation}

We observe $Z_1$, $\ldots$, $Z_n$ i.i.d. samples from $Z$ and let
\begin{equation}
V_n(x) := \frac{1}{n} \sum_{i=1}^n v(x, Z_i), \quad x \in \mathbb{R}^d.
\end{equation}
Under the assumption Theorem \ref{thm:main}(ii) stated right after, almost surely, for $n$ large enough $V_n$ is coercive and then let us define $X_n^\star \in \mathbb{R}^d$ to be an empiric minimizer of $V_n$ i.e.
\begin{equation}
X_n^\star \in \argmin(V_n) = \argmin_{x \in \mathbb{R}^d} \left( \frac{1}{n} \sum_{i=1}^n v(x,Z_i) \right).
\end{equation}
Our objective is to prove that minimizing $V_n$ the empirical observation of $V$ yields a good approximation of the true minimum $V^\star$, i.e. we aim at proving that $V_n(X_n^\star)$ converges to $V^\star$ and at giving bounds on $\mathbb{E}|V_n(X_n^\star)-V^\star|$.
For convenience we also define the convergence rate depending on the dimension, equal to $n^{-1/q}$ in general:
$$ \mathcal{R}_q(n) = \left\lbrace \begin{array}{ll}
n^{-1/2} & \text{if } q = 1, \\
n^{-1/2}\log(1+n) & \text{if } q = 2, \\
n^{-1/q} & \text{if } q > 2. \end{array} \right. $$
We define the moment of $\mu$ of order $r>0$:
$$ \mathscr{M}_r(\mu) := \int_{\mathbb{R}^q} |x|^r \mu(dx) \in [0,\infty] .$$

We now state our main results.
\begin{theorem}
\label{thm:main}
Assume that:
\begin{enumerate}[label=(\roman*)]
	\item The distribution $\mu$ satisfies $\mathscr{M}_r(\mu) < \infty$ for some $r > q/(q-1)$ if $q \ne 1$, for some $r>2$ if $q=1$.
	\item There exists a compact set $\mathcal{Q} \subset \mathbb{R}^q$ such that $x \mapsto \inf_{z \in \mathcal{Q}} v(x,z)$ is coercive and $\mathbb{P}(Z \in \mathcal{Q}) > 0$.
	\item For all $M>0$, there exists a constant $C$ such that for all compact set $\mathcal{K} \subset \mathbb{R}^q$:
	\begin{align}
	\label{eq:hyp1}
	& \sup_{x \in \textbf{B}(0,M)} \ \mathbb{E} [v(x,Z) \mathds{1}_{Z \notin \mathcal{K}} ] \le C e^{-C \textup{Radius}(\mathcal{K})^2} \\
	\label{eq:hyp2}
	& \sup_{x \in \textbf{B}(0,M)} \ [v(x,\cdot)_{|\mathcal{K}}]_{\textup{Lip}} \le C e^{C \textup{Radius}(\mathcal{K})}.
	\end{align}
	\item There exists $x^\star \in \argmin(V)$ such that $\mathbb{E}[v^2(x^\star,Z)] < \infty$.
\end{enumerate}
Then
\begin{enumerate}[label=(\alph*)]
	\item $V_n(X_n^\star)$ converges to $V^\star$ in $L^1(\mathbb{P})$ and 
	\begin{equation}
	\mathbb{E}|V_n(X_n^\star) - V^\star| \le Ca_n \mathcal{R}_q(n),
	\end{equation}
	\item If we assume furthermore that $\argmin(V) = \lbrace x^\star \rbrace$ is reduced to one point, that $\nabla^2 V(x^\star)$ exists and is positive definite and that $(\mathbb{E}|X_n|^2)$ is bounded, then $X_n^\star$ converges to $x^\star$ in $L^2(\mathbb{P})$ and
	\begin{equation}
	\mathbb{E}|X_n^\star - x^\star|^2 \le Ca_n \mathcal{R}_q(n) ,
	\end{equation}
\end{enumerate}
where $(a_n)$ is some positive sequence such that $a_n = o(n^{\varepsilon})$ for every $\varepsilon > 0$.
\end{theorem}

\begin{remark}
The condition that $(\mathbb{E}|X_n|^2)$ is bounded is satisfied under the assumption that there exist $A \ge 0$ and $B > 0$ such that $V(x) \ge A+B|x|^2$. Indeed, we then have
$$ \mathbb{E}|X_n|^2 \le B^{-1}(\mathbb{E}V(X_n^\star) - A) = B^{-1}(\mathbb{E}V_n(X_n^\star) - A) ,$$
which is bounded using Theorem \ref{thm:main}(a).
\end{remark}

\section{Applications}
\label{sec:applications}

\subsection{Importance sampling for Monte Carlo simulation}
\label{subsec:black_scholes}

Let us introduce a problem from \cite{lemaire2010}. Let $\varphi : \mathbb{R}^d \to \mathbb{R}$ be $\mathcal{C}^0$ and let $Z \sim \mathcal{N}(0,I_d)$. In order to estimate $\mathbb{E}[\varphi(Z)]$, we apply a variance reduction by importance sampling on the variable $\varphi(Z)$. For example, for Call and Put options on a Black-Scholes model, we have respectively
\begin{align}
\label{eq:def_call}
& \varphi(Z) = \left( \sum_{i=1}^{d} a_{i} s_{0}^{i} \exp \left[ \left( r - \sum_{j} \frac{\sigma_{i,j}^2}{2} \right) T + \sqrt{T} \sum_{j} \sigma_{i,j} Z_{j} \right] - K \right)_{+}, \\
\label{eq:def_put}
& \varphi(Z) = \left( K - \sum_{i=1}^{d} a_{i} s_{0}^{i} \exp \left[ \left( r - \sum_{j} \frac{\sigma_{i,j}^2}{2} \right) T + \sqrt{T} \sum_{j} \sigma_{i,j} Z_{j} \right] \right)_{+},
\end{align}
where $r$, $T$, $K \in \mathbb{R}^+$, $\sigma \in \mathcal{M}_d(\mathbb{R})$ is symmetric definite positive and $s_0$, $a \in \mathbb{R}^{d}$. Let us apply an
importance sampling by translation. For $x \in \mathbb{R}^d$ we have
\begin{align}
\mathbb{E}[\varphi(Z)] & = \dfrac{1}{(2\pi)^{d/2}} \int_{\mathbb{R}^{d}} \varphi(z) e^{-\frac{|z|^{2}}{2}} dz = \dfrac{1}{(2\pi)^{d/2}} \int_{\mathbb{R}^{d}} \varphi(z+x) e^{\frac{-|x|^{2}}{2} - \langle x, z \rangle } e^{-\frac{z^{2}}{2}} dz \nonumber \\
\label{eq:black_scholes_expectation}
& = \mathbb{E}\left[ \varphi(Z+x) e^{\frac{-|x|^{2}}{2} - \langle x, Z \rangle } \right]
\end{align}
and we have
\begin{align*}
\text{Var} \left[ \varphi(Z+x) e^{\frac{-|x|^{2}}{2} - \langle x, Z \rangle } \right] & = \mathbb{E} \left[ \varphi^{2}(Z+x) e^{-|x|^{2} - 2\langle x, Z \rangle } \right] - \mathbb{E} \left[ \varphi(Z+x) e^{\frac{-|x|^{2}}{2} - \langle x, Z \rangle } \right]^{2} \\
& = \mathbb{E} \left[ \varphi^{2}(Z+x) e^{-|x|^{2} - 2\langle x, Z \rangle } \right] - \mathbb{E}[\varphi(Z)]^{2} .
\end{align*}
The objective is then
\begin{equation}
\underset{x \in \mathbb{R}^d}{\text{Minimize}} \ V(x) = \mathbb{E} \left[ \varphi^{2}(Z+x) e^{-|x|^{2} - 2\langle x, Z \rangle } \right] = \mathbb{E} \left[ \varphi^{2}(Z) e^{-\langle x, Z \rangle + |x|^2/2 } \right] =: \mathbb{E}[v(x,Z)].
\end{equation}
in order to estimate $\mathbb{E}[\varphi(Z)]$ using \eqref{eq:black_scholes_expectation} by Monte Carlo simulation with a smaller variance.

\begin{proposition}
Let $Z \sim \mathcal{N}(0,I_d)$ and let $\varphi : \mathbb{R}^d \to \mathbb{R}$ be $\mathcal{C}^0$ such that there exist $A$, $B \ge 0$ such that
\begin{equation}
\label{eq:hyp_varphi}
\varphi \not\equiv 0, \quad \forall z \in \mathbb{R}^d, \ |\varphi(z)| \le Ae^{B|z|} \quad \text{and} \quad \forall r >0, \ [\varphi_{|\textbf{B}(0,r)}]_{\textup{Lip}} \le Ae^{Br}.
\end{equation}
Then the function $v : \mathbb{R}^d \times \mathbb{R}^d \to \mathbb{R}^+$ defined by
$$ v(x,z) = \varphi^2(z) e^{-\langle x,Z \rangle + |x|^2/2} $$
satisfies the assumptions of Theorem \ref{thm:main} and in particular:
$$ \mathbb{E}|V_n(X_n^\star) - V^\star| \le Ca_n \mathcal{R}_d(n). $$
\end{proposition}
\begin{proof}
$\bullet$ Since $\varphi \not\equiv 0$, there exists a compact set $\mathcal{Q} \subset \mathbb{R}^q$ with non empty interior such that $\inf_{z \in \mathcal{Q}} \varphi(z) > 0$ and since $Z \sim \mathcal{N}(0,I_q)$, we have $\mathbb{P}(Z \in \mathcal{Q}) > 0$.

$\bullet$ For $\mathcal{K} = \textbf{B}(0,R)$ we have
\begin{align*}
& \mathbb{E}\left[\sup_{x \in \textbf{B}(0,M)} v(x,Z) \mathds{1}_{Z \notin \mathcal{K}} \right] = C \int_{\textbf{B}(0,R)^c} \sup_{x \in \textbf{B}(0,M)} \varphi^2(z) e^{-\langle x,z \rangle + |x|^2/2 - |z|^2/2} dz \\
& \quad \le C \int_{\textbf{B}(0,R)^c} e^{2B|z| +M|z| + |M|^2/2 - |z|^2/2} dz \le C \int_{\textbf{B}(0,R)^c} e^{C|z|-|z|^2/2} dz \le Ce^{-C'R^2}.
\end{align*}

$\bullet$ For $\mathcal{K} = \textbf{B}(0,R)$ and for $x \in \textbf{B}(0,M)$ we have
\begin{align*}
[v(x,\cdot)_{|\mathcal{K}}]_{\text{Lip}} & \le \sup_{z \in \mathcal{K}} \left( \varphi^2(z)|x|e^{|x||z|} + 2\varphi(z) [\varphi_{\mathcal{K}}]_{\text{Lip}} e^{|x||z|} \right) \\
& \le Ce^{2BR+MR} + Ce^{(B+B')R+MR} \le Ce^{C'R} .
\end{align*}
\end{proof}

\begin{remark}
If $\varphi$ is defined as a Black-Scholes Call \eqref{eq:def_call} or Put \eqref{eq:def_put} then it satisfies the assumption \eqref{eq:hyp_varphi}.
\end{remark}

\subsection{Neural Networks}

Let us consider a regression problem with a fully connected neural network with quadratic loss and quadratic regularization. Let $\varphi : \mathbb{R} \rightarrow \mathbb{R}$ be the sigmoid function $x \mapsto (1+e^{-x})^{-1}$.
Let $K \in \mathbb{N}$ be the number of layers and for $k=1$, $\ldots$, $K$, let $d_k \in \mathbb{N}$ be the size of the $k^{\text{th}}$. For $u \in \mathbb{R}^{d_{k-1}}$ and for $x \in \mathcal{M}_{d_k,d_{k-1}}(\mathbb{R})$, we define $\varphi_{x}(u) := [\varphi([x \cdot u]_i)]_{1 \le i \le d_k}$. The output of the neural network is
\begin{align*}
& \psi : \mathbb{R}^{d_1,d_0} \times \cdots \times \mathbb{R}^{d_K,d_{K-1}} \times \mathbb{R}^{d_0} \to \mathbb{R}^{d_K} \\
& \psi (x_1,\ldots,x_K,u) = \psi(x,u) = x_K \cdot \varphi_{x_{K-1}} \circ \ldots \circ \varphi_{x_1}(u).
\end{align*}
Let $u_i \in \mathbb{R}^{d_0}$ and $y_i \in \mathbb{R}^{d_K}$ be the data for $1 \le i \le N$. The objective is
$$ \underset{x_1, \ldots, x_K}{\text{minimize}} \quad V(x) := \frac{1}{2N} \sum_{i=1}^N |\psi(x_1, \ldots, x_K,u_i) - y_i|^2 + \frac{\lambda}{2}|x|^2 ,$$
where $x = (x_1, \ldots, x_K)$ and where $\lambda > 0$ is the regularization parameter. We denote $d := d_0 d_1 + \cdots + d_{K-1}d_K$ the dimension of the optimization problem.

\begin{proposition}
Let $\mu := (1/N) \sum_{i=1}^N \delta_{u_i,y_i}$ be the probability measure on $\mathbb{R}^{d_0} \times \mathbb{R}^{d_K}$ associated to the data, let $Z \sim \mu$ and let
$$ v : \mathbb{R}^d \times \mathbb{R}^{d_0} \times \mathbb{R}^{d_K} \to \mathbb{R}^+, \quad (x,u,y) \mapsto \frac{1}{2} |\psi(x,u)-y|^2 + \frac{\lambda}{2} |x|^2 .$$
Then $Z$ and $v$ satisfy the assumptions of Theorem \ref{thm:main}.
\end{proposition}
\begin{proof}
$\bullet$ The condition $(ii)$ is satisfied taking any compact set $\mathcal{Q} \subset	\mathbb{R}^{d_0} \times \mathbb{R}^{d_K}$ containing at least one data $(u_i,y_i)$ and noting that $v(x,u,y) \ge (\lambda/2)|x|^2$.

$\bullet$ The condition \eqref{eq:hyp1} is satisfied because $Z$ has compact support.

$\bullet$
We have
$$ \partial_y v(x,u,y) = |\psi(x,u)-y| \le |y| + \|x_K\| $$
and
\begin{align*}
\partial_u v(x,u,y) = |\psi(x,u)-y| \cdot \|\partial_u \psi(x,u)| \le (|y| + \|x_K\|) \cdot \|x_K\| \cdot |\varphi'_{x_{K-1}} \circ \varphi_{x_{K-2}} \cdots \circ \varphi_{x_1}(u)| \cdots |\varphi'_{x_1}(u)|,
\end{align*}
and since all the functions $\varphi'_{x_k}$ are bounded for $x \in \textbf{B}(0,M)$, we obtain \eqref{eq:hyp2}.
\end{proof}


\section{Proof of Theorem \ref{thm:main}}
\label{sec:proof}

\subsection{Convergence of the empirical measure for the $L^1$-Wasserstein distance}

Let $\mu$ be a probability distribution on $\mathbb{R}^q$ and let $(Z_n)_{n \ge 1}$ be an i.i.d. sequence of random variables of law $\mu$. For $n \ge 1$, let us denote the empirical measure
$$ \mu_n := \frac{1}{n} \sum_{i=1}^n \delta_{Z_i} .$$
\cite[Theorem 1]{fournier2014} gives the rate of convergence for the $L^p$-Wasserstein distance of $\mu_n$ to $\mu$.
\begin{theorem}
\label{thm:fournier}
Let $p>0$ and assume that for some $r>p$, $\mathscr{M}_r(\mu) < \infty$. Then there exists a constant $C>0$ depending only on $q$, $p$, $r$ such that for all $n \ge 1$,
$$ \mathbb{E}[\mathcal{W}_p(\mu_n, \mu)] \le C\mathscr{M}_r^{p/r}(\mu) \left\lbrace \begin{array}{ll} n^{-1/2} + n^{-(r-p)/r} & \text{if } p > q/2 \text{ and } r \ne 2p, \\
n^{-1/2}\log(1+n) + n^{-(r-p)/r} & \text{if } p=q/2 \text{ and } r \ne 2p, \\
n^{-p/q} + n^{-(r-p)/r} & \text{if } p \in (0,q/2) \text{ and } r \ne q/(q-p).
\end{array} \right. $$
where the expectation is taken on the samples $X_1, \ldots, X_n$.
\end{theorem}
Right after \cite[Theorem 1]{fournier2014} are given right examples proving that these bounds are sharp, which indicates us that we cannot get better bounds in Theorem \ref{thm:main} in general.

We now consider the case $p=1$. Following the assumption Theorem \ref{thm:main}(i) stating that $\mathscr{M}_r(\mu)<\infty$ for some $r > q/(q-1)$ if $q \ne 1$ and for some $r>2$ if $q = 1$, the term $n^{-(r-1)/1}$ becomes negligible and then
$$ \mathbb{E}[\mathcal{W}_1(\mu_n, \mu)] \le C_{\mu} \mathcal{R}_q(n) .$$
Let us recall the Kantorovich-Rubinstein representation of the $L^1$-Wasserstein distance \cite[Equation (6.3)]{villani2009}:
\begin{equation}
\label{eq:kantorovich}
\mathcal{W}_1(\mu_1, \mu_2) = \sup \left\lbrace \int_{\mathbb{R}^q} f(x)(\mu_1 - \mu_2)(dx) : \ f : \mathbb{R}^d \to \mathbb{R}, \ [f]_{\text{Lip}}=1 \right\rbrace.
\end{equation}

\subsection{Proof of Theorem \ref{thm:main}(a)}

\begin{lemma}
\label{lemma:1}
Assume that $X_n^\star \in \Theta$ and that $x^\star \in \Theta$ for some $x^\star \in \argmin(V)$ and for some set $\Theta \subset \mathbb{R}^d$. Then we have
$$ |V_n(X_n^\star) - V^\star| \le \sup_{x \in \Theta} |V_n(x) - V(x)| .$$
\end{lemma}
\begin{proof}
If $V_n(X_n^\star) \ge V^\star$ then
$$ |V_n(X_n^\star) - V^\star| = V_n(X_n^\star) - V^\star \le V_n(x^\star) - V(x^\star) = |V_n(x^\star) - V(x^\star)| .$$
If $V_n(X_n^\star) \le V^\star$ then
$$ |V_n(X_n^\star) - V^\star| = V(x^\star) - V_n(X_n^\star) \le V(X_n^\star) - V_n(X_n^\star) = |V(X_n^\star) - V_n(X_n^\star)| .$$
\end{proof}

From now, our strategy of proof will be the following. If $X_n^\star$ and $x^\star$ are in some set $\Theta \subset	\mathbb{R}^d$, then using \eqref{eq:kantorovich} we have:
\begin{align*}
\mathbb{E}|V_n(X_n^\star) - V^\star| & \le \mathbb{E} \sup_{x \in \Theta} |V_n(x) - V(x)| \\
& = \mathbb{E} \sup_{x \in \Theta}\left| \int_{\mathbb{R}^q} v(x,z)\mu_n(dz) - \int_{\mathbb{R}^q} v(x,z) \mu(dz)\right| \\
& \le \left( \sup_{x \in \Theta} [v(x,\cdot)]_{\text{Lip}} \right) \mathbb{E}[\mathcal{W}_1(\mu_n,\mu)].
\end{align*}
So if $v(x,\cdot)$ is Lipschitz-continuous and uniformly for $x \in \Theta$, then we can directly apply Theorem \ref{thm:fournier} to directly get a bound for $\mathbb{E}|V_n(X_n^\star) - V^\star|$. However as in Section \ref{subsec:black_scholes}, this is not the case in general.

We first prove we can consider that $X_n^\star \in \Theta$ where $\Theta = \textbf{B}(0,M)$ for some $M>0$.
\begin{lemma}
\label{lemma:2}
Let $M_0>0$ such that $x^\star \in \textbf{B}(0,M_0)$ and
$$ \inf_{x \notin \textbf{B}(0,M_0)} \inf_{z \in \mathcal{Q}} v(x,z) > 4V^\star / \mathbb{P}(Z \in \mathcal{Q}) .$$
Note that such $M_0$ exists using the assumption Theorem \ref{thm:main}(ii). Then
\begin{equation}
\mathbb{E}|(V_n(X_n^\star)-V^\star) \mathds{1}_{|X_n^\star|>M}| \le Cn^{-1}.
\end{equation}
\end{lemma}
\begin{proof}
We define $\kappa := \mathbb{P}(Z \in \mathcal{Q}) > 0$. We have
\begin{align*}
\mathbb{E}|(V_n(X_n^\star)-V^\star) \mathds{1}_{|X_n^\star|>M_0}| & \le \mathbb{E}[V_n(X_n^\star)\mathds{1}_{|X_n^\star|>M_0}] + V^\star \mathbb{P}\left(|X_n^\star|>M_0 \right).
\end{align*}
But if $|X_n^\star|>M_0$ then
$$ \inf_{|x|>M_0} V_n(x) \le V_n(X_n^\star) \le V_n(x^\star) ,$$
so necessarily we cannot have 
$$ V_n(x^\star) < 2V^\star \le \inf_{|x| > M_0} V_n(x) ,$$
so that
\begin{align}
\label{eq:proof:21}
\mathbb{P}\left(|X_n^\star|> M_0 \right) & \le \mathbb{P}(V_n(x^\star) \ge 2V^\star) + \mathbb{P}\left(\inf_{|x|>M_0} V_n(x) < 2V^\star \right) .
\end{align}
By the Tchebycheff inequality, we have
$$ \mathbb{P}(V_n(x^\star) \ge 2V^\star) \le  \mathbb{P}(|V_n(x^\star)-V^\star| \ge V^\star) \le \frac{\text{Var}[v(x^\star,Z)]}{V^\star n} .$$
And if $\inf_{|x|>M_0} V_n(x) < 2V^\star$ then using that $v \ge 0$ we have
$$ 2V^\star > \inf_{|x|>M_0} V_n(x) \ge \inf_{|x|>M_0} \frac{1}{n} \sum_{i=1}^n v(x,Z_i) \mathds{1}_{Z_i \in \mathcal{Q}} \ge \left( \inf_{|x| > M_0} \ \inf_{z \in \mathcal{Q}} v(x, z) \right) \frac{1}{n} \sum_{i=1}^n \mathds{1}_{Z_i \in \mathcal{Q}} ,$$
so necessarily $(1/n) \sum_{i=1}^n \mathds{1}_{Z_i \in \mathcal{Q}} < \kappa/2$ so using the Tchebycheff inequality again:
\begin{align*}
\mathbb{P}\left(\inf_{|x|>M_0} V_n(x) < 2V^\star \right) & \le \mathbb{P}\left(\frac{1}{n} \sum_{i=1}^n \mathds{1}_{Z_i \in \mathcal{Q}} < \frac{\kappa}{2} \right) \le \mathbb{P}\left(\left|\frac{1}{n} \sum_{i=1}^n \mathds{1}_{Z_i \in \mathcal{Q}} - \kappa \right| > \frac{\kappa}{2} \right) \le \frac{4}{n\kappa^2}.
\end{align*}
So we obtain from \eqref{eq:proof:21} that
\begin{equation}
\label{eq:proba_compact}
\mathbb{P}(|X_n^\star| > M_0) \le C/n
\end{equation}
and then
\begin{align*}
\mathbb{E}|(V_n(X_n^\star)-V^\star) \mathds{1}_{|X_n^\star|>M_0}| & \le \mathbb{E}|V_n(X_n^\star)\mathds{1}_{|X_n^\star|>M_0}| + V^\star \mathbb{P}\left(|X_n^\star|>M_0 \right) \\
& \le C\sqrt{\mathbb{E}[V_n^2(X_n^\star)]}n^{-1/2} + Cn^{-1} \\
& \le C \text{Var}[v(x^\star,Z)] n^{-1} + Cn^{-1}.
\end{align*}
\end{proof}


We now prove Theorem \ref{thm:main}.
\begin{proof}
Let us define the compact sets:
$$ \Theta = \textbf{B}(0,M_0) \subset \mathbb{R}^d, \quad \mathcal{K}_n := \textbf{B}(0,\beta_n) \subset \mathbb{R}^q, \quad \widetilde{\mathcal{K}}_n := \textbf{B}(0,\beta_n+1) \subset \mathbb{R}^q ,$$
where $(\beta_n)$ is a positive increasing sequence that we shall define later. For $n \in \mathbb{N}$ we also define the continuous function
$$ \phi_n : \mathbb{R}^q \to \mathbb{R}, \quad z \mapsto \left\lbrace \begin{array}{ll}
1 & \text{if } z \in \mathcal{K}_n, \\
0 & \text{if } z \notin \widetilde{\mathcal{K}_n}, \\
\beta_n + 1 - |z| & \text{if } z \in \widetilde{\mathcal{K}_n} \setminus \mathcal{K}_n.
\end{array} \right. $$
Then we have
\begin{align}
\label{eq:proof:11}
\mathbb{E}|V_n(X_n^\star)-V^\star| & = \mathbb{E}|(V_n(X_n^\star)-V^\star) \mathds{1}_{|X_n^\star|>M_0}| + \mathbb{E}|(V_n(X_n^\star)-V^\star)\mathds{1}_{|X_n^\star| \le M_0}|.
\end{align}
Using Lemma \ref{lemma:2}, the first term is bounded by $Cn^{-1}$.
We now turn to the second term of \eqref{eq:proof:11}:
\begin{align}
\sup_{x \in \Theta} |V_n(x) - V(x)| & = \sup_{x \in \Theta} \left| \int_{\mathbb{R}^q} v(x,z) \mu_n(dz) - \int_{\mathbb{R}^q} v(x,z) \mu(dz) \right| \nonumber \\
& = \sup_{x \in \Theta} \left| \int_{\mathbb{R}^q} v(x,z) \phi_n(z) (\mu_n-\mu)(dz) + \int_{\mathbb{R}^q} v(x,z)(1-\phi_n(z)) (\mu_n-\mu)(dz) \right| \nonumber \\
& \le \left( \sup_{x \in \Theta} [(v(x,\cdot) \phi_n)]_{\text{Lip}} \right) \mathcal{W}_1(\mu_n,\mu) + \sup_{x \in \Theta} \int_{z \notin \mathcal{K}_n} v(x,z) \mu(dz) \nonumber \\
\label{eq:proof:1}
& \quad + \sup_{x \in \Theta} \int_{z \notin \mathcal{K}_n} v(x,z) \mu_n(dz).
\end{align}
Using the assumption \eqref{eq:hyp1}, we have
$$ \sup_{x \in \Theta} \int_{z \notin \mathcal{K}_n} v(x,z) \mu(dz) = \sup_{x \in \Theta} \mathbb{E}[v(x,Z) \mathds{1}_{z \notin \mathcal{K}_n}] \le Ce^{-C\beta_n^2}. $$
Using the assumption \eqref{eq:hyp2}, we have
$$ \sup_{x \in \Theta} [(v(x,\cdot) \phi_n)]_{\text{Lip}} \le Ce^{C\beta_n} .$$

Then taking the expectation over $Z_1$, $\ldots$, $Z_n$ and remarking that
$$ \mathbb{E}\left[\sup_{x \in \Theta} \int_{z \notin \mathcal{K}_n} v(x,z) \mu_n(dz) \right] = \sup_{x \in \Theta} \int_{z \notin \mathcal{K}_n} v(x,z) \mu(dz), $$
using Lemma \ref{lemma:1} and Theorem \ref{thm:fournier} we obtain
\begin{align*}
\mathbb{E}|V_n(X_n^\star)-V^\star| & \le \mathbb{E}\sup_{x \in \Theta} |V_n(x) - V(x)| \le Ce^{C\beta_n} \mathbb{E}[\mathcal{W}_1(\mu_n,\mu)] + Ce^{-C\beta_n^2} \le Ce^{C\beta_n}n^{-1/q} + Ce^{-C\beta_n^2}.
\end{align*}
We now choose $\beta_n = D\sqrt{\log(n)}$ with $D>0$ so that
$$ \mathbb{E}|V_n(X_n^\star)-V^\star| \le Ce^{C'\sqrt{\log(n)}}n^{-1/q} + Cn^{-CD^2} .$$
Choosing $D>0$ large enough yields
$$ \mathbb{E}|V_n(X_n^\star)-V^\star| \le Ce^{C'\sqrt{\log(n)}}n^{-1/q} .$$

\end{proof}

\subsection{Proof of Theorem \ref{thm:main}(b)}

\begin{proof}
We define $\Theta = \textbf{B}(0,M_0) \subset \mathbb{R}^d$. Since $\nabla^2 V(x^\star)$ is definite positive, there exists $r>0$ such that for all $x \in \mathbf{B}(x^\star,r)$, $\nabla^2 V(x)$ is definite positive and we define
$$ \lambda_{\min} := \min \left\lbrace \lambda : \ \lambda \text{ eigenvalue of } \nabla^2 V(x), \ x \in \mathbf{B}(x^\star,r) \right\rbrace > 0 .$$
Then we have
\begin{align*}
\mathbb{E}|X_n^\star - x^\star|^2 & = \mathbb{E}|(X_n^\star - x^\star)\mathds{1}_{X_n^\star \notin \Theta}|^2 + \mathbb{E}|(X_n^\star - x^\star)\mathds{1}_{X_n^\star \in \Theta \cap \textbf{B}(x^\star,r)^c}|^2 + \mathbb{E}|(X_n^\star - x^\star)\mathds{1}_{X_n^\star \in \Theta \cap \textbf{B}(x^\star,r)}|^2 \\
& =: E_1 + E_2 + E_3.
\end{align*}
And using \eqref{eq:proba_compact} we have
$$ E_1 \le (\mathbb{E}|X_n^2| + |x^\star|) \mathbb{P}(X_n^\star \notin \Theta) \le C/n $$
For the second term, since $x^\star$ is the unique minimizer of $V$, we have $\beta := \inf_{x \notin \mathbf{B}(x^\star,r)} (V(x) - V(x^\star)) > 0$ and then
\begin{align*}
E_2 & \le \left(M_0+|x^\star|\right)^2 \mathbb{P}(X_n^\star \notin \textbf{B}(x^\star,r)) \le \frac{\left(M_0+|x^\star|\right)^2}{\beta} \mathbb{E}[ (V(X_n^\star)-V^\star) \mathds{1}_{X_n^\star \notin \textbf{B}(x^\star,r)}] \\
& \le C\mathbb{E}[V(X_n^\star)-V(x^\star)] = C\mathbb{E}[V_n(X_n^\star)-V(x^\star)] \le Ca_n \mathcal{R}_q(n).
\end{align*}

For the third term, by a Taylor expansion there exists some $\xi_n \in (X_n^\star, x^\star)$ such that
$$ V(X_n^\star) = V(x^\star) + \frac{1}{2} \nabla^2 V(\xi_n) \cdot (X_n^\star - x^\star)^{\otimes 2} .$$
So that
\begin{align*}
\nabla^2 V(\xi_n) \cdot (X_n^\star - x^\star)^{\otimes 2} & = 2(V(X_n^\star)-V(x^\star)) \le 2(V(X_n^\star) - V_n(X_n^\star) + V_n(x^\star) - V(x^\star)) \\
& \le 4 \sup_{x \in \Theta} |V_n(x) - V(x)|
\end{align*}
and then
\begin{align*}
E_3 \le \mathbb{E} \left[\nabla^2 V(\xi_n) \cdot (X_n^\star - x^\star)^{\otimes 2} \right] \le \frac{4}{\lambda_{\min}} \mathbb{E} \sup_{x \in \Theta} |V_n(x) - V(x)| \le C a_n \mathcal{R}_q(n).
\end{align*}
\end{proof}


\end{document}